\documentclass[12pt]{amsart}
\usepackage{latexsym,amsfonts,amssymb,amsmath,amsthm, url, hyperref,graphicx, us} 
\usepackage[inner=1.0in,outer=1.0in,bottom=1.0in, top=1.0in]{geometry}

\newcommand{\shortmod}{\ensuremath{\negthickspace \negthickspace \negthickspace \pmod}}
\newcommand{\dedekindsum}{S_{\chi_1, \chi_2}}
\newtheorem{conditions}{Condition}

\newcommand{\floor}[1]{\left\lfloor {#1} \right\rfloor}
\newcommand{\fractional}[1]{\left\{ {#1} \right\}}
\newcommand{\parentheses}[1]{\left( {#1} \right)}
\newcommand{\brackets}[1]{\left[ {#1} \right]}
\newcommand{\mycomment}[1]{}

\title{The image of the generalized Dedekind sum}

\author[E. Knight]{Evelyne S. Knight}
 \address{Department of Mathematics \\
        Pomona College \\
        Claremont \\
        CA 91711 \\
         U.S.A}
 \email{eski2022@mymail.pomona.edu}

\author[C. Matos]{Carlos Alexov Matos}
 \address{Department Of Mathematics \\
 Michigan State University \\
 East Lansing \\
 MI 48824 \\
 U.S.A.
 }
 \email{matoscar@msu.edu}

\author[A. Sefidi]{Amira Sefidi}
 \address{Department Of Mathematics \\
 The University of Texas at Austin \\
 Austin \\
 TX 78712 \\
 U.S.A.
 }
 \email{amira.sefidi@utexas.edu}

\author[M. Young]{Matthew P. Young}
 \address{Department of Mathematics \\
 	  Rutgers University \\
 	  Piscataway \\
 	  NJ 08854 \\
 		U.S.A.}		
 \email{mpy4@math.rutgers.edu}

\date{}

\thanks{This material is based upon work supported by the National Science Foundation under agreement No. DMS-2302210 (M.Y.).  Any opinions, findings and conclusions or recommendations expressed in this material are those of the authors and do not necessarily reflect the views of the National Science Foundation. }

\begin{document}

\begin{abstract}
    The newform Dedekind sum $S_{\chi_1, \chi_2}$ associated to a pair of primitive Dirichlet characters $\chi_1$, $\chi_2$ of respective conductors $q_1$, $q_2$,
is a group homomorphism from $\Gamma_1(q_1 q_2)$ into the number field $F_{\chi_1, \chi_2}$ generated by the values of the characters. 
It is a basic question to identify the image of this map, which is known to be a lattice $L_{\chi_1, \chi_2}$ in $F_{\chi_1, \chi_2}$.
It has recently been conjectured that when $\chi_1$ and $\chi_2$ are quadratic, then $ L_{\chi_1, \chi_2} = 2 \mathbb{Z}$.  
    In this paper, we make some progress towards this conjecture by exhibiting an explicit lattice in which $L_{\chi_1, \chi_2}$ is contained; in particular, when the characters are quadratic, the $q_i$ are coprime, odd, and sufficiently large, then $L_{\chi_1, \chi_2} \subseteq \mathbb{Z}$.
\end{abstract}

\maketitle

\section{Introduction}
Historically, the classical Dedekind sum  has found applications across diverse disciplines: from analytic number theory (notably Dedekind’s $\eta$-function) and combinatorial geometry (focusing on lattice point enumeration) to topology (studying signature defects of manifolds) and algebraic number theory (including class number formulas).

\begin{definition}
    For coprime integers $h, k$ where $k >0$, the classical Dedekind sum is defined as
    \[
    s(h,k) = \sum_{j \shortmod{k}} B_{1}\Big(\frac{j}{k}\Big)B_{1}\Big(\frac{hj}{k}\Big)
    \]
   where B$_{1}(x)$ is the first Bernoulli function, also known as the sawtooth function:
    \[
    B_{1}(x) = \begin{cases}
        0,\; \text{if} \; x \in \Z & \\
        x - \lfloor x \rfloor - \frac{1}{2},  \; \text{otherwise}.
    \end{cases}
    \]
    Let $ \left\{ x\right\} = x - \lfloor x \rfloor$ denote the fractional part of x.
\end{definition}

Many properties of the classical Dedekind sum can be derived by way of the Kronecker limit formula for the level $1$ real analytic Eisenstein series.
The newform Dedekind sum (also referred to as the generalized Dedekind sum) was 
studied in \cite{SVY} by extending the Kronecker limit formula for higher level newform Eisenstein series.
\begin{definition}[Newform Dedekind sum, \cite{SVY}]
\label{defi:SVY}
Let $\chi_{1}$ and $\chi_{2}$ be nontrivial, primitive Dirichlet characters modulo $q_{1}$ and $q_{2}$, respectively, where $\chi_{1}\chi_{2}(-1) = 1$. For $\gamma = \left(\begin{smallmatrix}
a & b \\
c & d
\end{smallmatrix}\right), $ $\gamma \in \Gamma_{0}(q_{1}q_{2})$ with $c \geq 1$, the generalized Dedekind sum associated to $\chi_{1}$ and $\chi_{2}$ is given by
\begin{equation}
\label{eq:DedekindDefBernoulli}
S_{\chi_{1}, \chi_{2}}(\gamma) = S_{\chi_1,\chi_2}(a,c)= \sum_{j \shortmod{c}}\sum_{n \shortmod{q_1}}\overline{\chi_2}(j)\overline{\chi_1}(n)B_1\Big(\frac{j}{c}\Big)B_1\Big(\frac{n}{q_1}+\frac{aj}{c}\Big).
\end{equation}
\end{definition}
The assumptions on $\chi_1$, $\chi_2$, $q_1$, and $q_2$ from Definition \ref{defi:SVY} will be in place for the rest of the paper.

In this paper, we are interested in understanding the image of $S_{\chi_1,\chi_2
}$. 
The set of values taken by the classical Dedekind sum has been extensively studied. 
In \cite[Theorem 2]{RG}, for example, the denominator of the classical sum is proven to be a divisor of $2c\cdot \gcd(3, c)$. 

\begin{definition}
    Let $F_{\chi_1,\chi_2}$ denote the smallest number field in which $\chi_1$ and $\chi_2$ take values.
    In addition, let $\Z[\chi_1,\chi_2]$ be the ring of integers of $F_{\chi_1,\chi_2}$.
\end{definition}

\begin{proposition}[\cite{Majure}]
\label{prop:majure-lattices}
     The image $S_{\chi_1, \chi_2}(\Gamma_1(q_1q_2))$ is a lattice (of full rank) inside $F_{\chi_1,\chi_2}$. That is,
    \[
    S_{\chi_1,\chi_2}(\Gamma_1(q_1q_2))=\oplus^n_{i=1}\alpha_i\mathbb{Z},
    \]
    for some $\alpha_1, \dots, \alpha_n \in F_{\chi_1,\chi_2}$, where n is the degree of the field extension $F_{\chi_1,\chi_2}$/ $\mathbb{Q}$.
\end{proposition}

In addition, other authors have focused specifically on the case when $\chi_1,\chi_2$ are quadratic. After extensive calculations, De Leon and McCormick posed the following conjecture.
\begin{conjecture}[The Two Conjecture, \cite{DM}]
\label{conj:2-conjecture}
Suppose $\chi_1$ and $\chi_2$ are quadratic. Then
\[
S_{\chi_1, \chi_2}(\Gamma_1(q_1 q_2))= 2\mathbb{Z}.
\]
\end{conjecture}

Conjecture \ref{conj:2-conjecture} is equivalent to the two set containments
\begin{align}
        \label{eqn:set one}
        S_{\chi_1, \chi_2}&(\Gamma_1(q_1 q_2)) \subseteq 2 \Z,\\
        \label{eqn:set two}
        2 \Z \subseteq &S_{\chi_1, \chi_2}(\Gamma_1(q_1 q_2)).
\end{align}
 We present two advancements towards proving \eqref{eqn:set one}, both of which lie in a more general setting than Conjecture~\ref{conj:2-conjecture}.

\begin{theorem}
\label{thm:main-thm-gamma1}
Let $\chi_1,\chi_2$ be Dirichlet characters of respective conductors $q_1,q_2$. Then
    \[ \dedekindsum(\Gamma_1(q_1 q_2)) \subseteq \frac{1}{\gcd(q_1,q_2)} \Z[\chi_1, \chi_2]. \]    
\end{theorem}

\begin{theorem}
    \label{thm:main-thm-gamma0} 
Let $\chi_1,\chi_2$ be quadratic and such that $q_1, q_2 > 4$ are odd. Then
    \[ \dedekindsum(\Gamma_0(q_1 q_2)) \subseteq \frac{1}{\gcd(q_1,q_2)} \Z. \]    
\end{theorem}

\begin{corollary}
    \label{cor:main-thm-with-two}
Let $\chi_1,\chi_2$ be Dirichlet characters such that $\gcd(q_1,q_2)=1$. 
Then 
\begin{align*}
    \dedekindsum(\Gamma_1(q_1 q_2)) &\subseteq \Z[\chi_1, \chi_2].
    \end{align*}
In addition, if $\chi_1, \chi_2$ are quadratic, 
with $q_1, q_2>4$ both odd,
then
\[
    \dedekindsum(\Gamma_0(q_1 q_2)) \subseteq \Z.
\]
\end{corollary}

Corollary ~\ref{cor:main-thm-with-two} is progress towards a more general version of \eqref{eqn:set one}. 

We computed some data to determine the image of the newform Dedekind sum in some cases where the characters are not necessarily quadratic, which is displayed in the following table\footnote{Code to compute this data was obtained via \cite{TW24}}. For each entry in the table, we calculated the image $S_{\chi_1, \chi_2}(\Gamma_1(q_1 q_2))$ for every Dirichlet character of the specified conductor and order.  In the table, we let $\omega = \exp(2 \pi i/3)$.
\[
 \begin{array}{||c|c|c|c||}
        \hline
        \chi_1: q_1, order&\chi_2: q_2, order&\text{Image}&F_{\chi_1,\chi_2}\\
        \hline
        3, 2&7, 2&2\Z&\Q\\
        \hline
        5, 2&8, 2&2\Z&\Q\\
        \hline
        5, 2&7, 3&2\Z[\omega]&\Q(\omega)\\
        \hline
        5, 4&5, 4&2\Z[i]&\Q(i)\\
        \hline
    \end{array}
\]
Based on these limited computations, we put forth the following
generalized two conjecture.
\begin{conjecture}[Generalized Two Conjecture]
\label{conj:generalized-2-conjecture}
    
    Suppose
    that $F_{\chi_1,\chi_2}$ is a quadratic number field. Then
    \[
    S_{\chi_1,\chi_2}(\Gamma_1(q_1q_2))=2\Z[\chi_1,\chi_2].
    \]
\end{conjecture}
Our restriction that $F_{\chi_1,\chi_2}$ is a quadratic number field is simply due to lack of data for higher degree extensions.
More extensive computations would be welcome.  It should be noted that De Leon and McCormick \cite{DM} computed the image for many more values of $q_1$ and $q_2$ than those listed in the above table, though only for quadratic characters.

\section{Preliminaries}
 First, we recall some of the basic facts about the newform Dedekind sums.
We define the following notation for use in future statements:
    \[\psi(\gamma) \coloneqq \chi_1\overline{\chi}_2(d)\text{, where }\gamma = (\begin{smallmatrix} a & b \\ c & d \end{smallmatrix}) \in \Gamma_0(q_1 q_2).\]

\begin{proposition}
[Crossed Homomorphism Property, \cite{SVY}]
    \label{prop:crossedhomomorphism}
Let $\gamma_1,\gamma_2 \in \Gamma_0(q_1q_2)$. Then
\begin{equation*}
S_{\chi_1,\chi_2}(\gamma_1\gamma_2)=S_{\chi_1,\chi_2}(\gamma_1)+\psi(\gamma_1)S_{\chi_1,\chi_2}(\gamma_2).
\end{equation*}
\end{proposition}

\begin{corollary}[Homomorphism Property, \cite{SVY}]
\label{cor:hom-property-gamma0}
The newform Dedekind sum
    \[S_{\chi_1, \chi_2} : \Gamma_1(q_1 q_2) \to \C \text{ is a group homomorphism.} \] 
\end{corollary}

We also have need of the following reciprocity formula.
\begin{proposition}
[Reciprocity Law, \cite{SVY}]
\label{prop:reciprocity-law}
For $\gamma = (\begin{smallmatrix} a & b \\  r q_1 q_2 & d \end{smallmatrix}) \in \Gamma_0(q_1 q_2)$, let $\gamma' = (\begin{smallmatrix} d & -r \\ -bq_1 q_2 & a \end{smallmatrix}) \in \Gamma_0(q_1 q_2)$.  If $\chi_1$ and $\chi_2$ are even, then
\begin{equation*}
S_{\chi_1, \chi_2}(\gamma) = S_{\chi_2, \chi_1}(\gamma').
\end{equation*}
If $\chi_1$ and $\chi_2$ are odd, then
\begin{equation*}
S_{\chi_1, \chi_2}(\gamma) = -S_{\chi_2, \chi_1}(\gamma') + (1 - \psi(\gamma)) \frac{\tau(\overline{\chi_1}) \tau(\overline{\chi_2})}{(\pi i)^2} L(1, \chi_1) L(1, \chi_2).
\end{equation*}
\end{proposition}
Note that if $\gamma$, $\gamma' \in \Gamma_1(q_1 q_2)$, then the reciprocity law simplifies to 
\begin{equation*}
S_{\chi_1, \chi_2}(\gamma) = \chi_1(-1)S_{\chi_2, \chi_1}(\gamma').
\end{equation*}

The following two results will help us re-write the newform Dedekind sum.
\begin{lemma}[{\cite[p. 94]{GKP}}]
    \label{lemma:floor-function-properties}
    Let $a, N$ be positive integers, and let $d = gcd(a, N)$. Then
    \[ \sum_{k = 0}^{N - 1} \floor{\frac{x + a k}{N}} = d \floor{\frac{x}{d}} + \frac{(a-1)(N-1)}{2} + \frac{d-1}{2}. \]
    In particular, if $a$ and $N$ are coprime, then
    \[
    \sum_{k = 0}^{N - 1} \floor{\frac{x + a k}{N}} = \floor{x} + \frac{(a-1)(N-1)}{2}.
    \]
\end{lemma}

\begin{lemma}
    \label{lemma:character-orthogonality}
    Suppose a function $f(m, n)$ is of the form $f(m, n) = f_1(m) + f_2(n)$. Then \[
    \sum_{m \shortmod {q_1}} \sum_{n \shortmod{q_2}} \chi_1(m) \chi_2(n) f(m, n) = 0. \]
\end{lemma}
\begin{proof}
    Substituting $f(m, n) = f_1(m) + f_2(n)$ and distributing, we get \[
    \brackets{\sum_{m \shortmod {q_1}} \chi_1(m) f_1(m) \sum_{n \shortmod{q_2}} \chi_2(n) } + \brackets{\sum_{n \shortmod {q_2}} \chi_2(n) f_2(n) \sum_{m \shortmod{q_1}} \chi_1(m) }.
    \]
    Both sums vanish by the orthogonality of Dirichlet characters.
\end{proof}

\begin{proposition}
For $c>0$ the following is an alternative formula for the Dedekind sum:
    \label{proposition:the-half-cancels}
    \begin{equation}
        \label{eqn:dedekind-sum-fractional}
        \dedekindsum(a,c) =\sum_{j \shortmod c} \sum_{n \shortmod{q_1}} \overline{\chi_2}(j) \overline{\chi_1}(n) \left\{ \frac{j}{c} \right\} \left\{\frac{aj}{c} + \frac{n}{q_1}\right\}.
    \end{equation}
\end{proposition}
To prove Proposition \ref{proposition:the-half-cancels}, we need the following:
\begin{lemma}
    \label{lemma:aj+nrq2-not-divis-by-c}
    Let $(\begin{smallmatrix}
        a&b\\c&d
    \end{smallmatrix})\in\Gamma_0(q_1q_2)$. 
    Then for $j,n\in \Z$, if $q_2\nmid j$, then
    \[
    \frac{aj}{c}+\frac{n}{q_1}\not\in \Z.
    \]
\end{lemma}

\begin{proof}
    Assume $\frac{aj}{c}+\frac{n}{q_1}\in \Z$, and write $c= rq_1 q_2$. Then $aj+nrq_2\equiv 0\pmod c$, so $\gcd(aj,c)=\gcd(j,c)=\gcd(nrq_2,c)$, which is a multiple of $q_2$.
Thus $q_2\mid j$.
\end{proof}

Now we can prove Proposition~\ref{proposition:the-half-cancels}.
\begin{proof}[Proof of Proposition~\ref{proposition:the-half-cancels}]
    Let us begin by analyzing the product
    \[
    B_1\parentheses{\frac{j}{c}} B_1\parentheses{\frac{aj}{c} + \frac{n}{q_1}} = \parentheses{\fractional{\frac{j}{c}} - \frac{1}{2}} \parentheses{\fractional{\frac{aj}{c} + \frac{n}{q_1}} - \frac{1}{2}}.
    \]
    Distributing leads to the following four terms:
    \[
    \fractional{\frac{j}{c}} \fractional{\frac{aj}{c} + \frac{n}{q_1}}
    - \frac{1}{2} \fractional{\frac{aj}{c} + \frac{n}{q_1}} - \frac{1}{2} \fractional{\frac{j}{c}} + \frac{1}{4}.
    \]
    Summing over $\chi_1$ and $\chi_2$ 
    as in Definition \ref{defi:SVY}, the last two terms vanish
    by Lemma \ref{lemma:character-orthogonality}. It is now sufficient to show that $Z=0$, where
    \begin{equation}
    \label{eqn:aj/c+n/q_1-vanishes}
Z=    \sum_{j \shortmod c} \sum_{n \shortmod{q_1}} \overline{\chi_2}(j) \overline{\chi_1}(n) \fractional{\frac{aj}{c} + \frac{n}{q_1}}.
    \end{equation}

    We use a change of variables $j \to -j$ and $n \to -n$ to get
    \[
    Z=\sum_{j \shortmod c} \sum_{n \shortmod{q_1}} \overline{\chi_2}(-j) \overline{\chi_1}(-n) {\fractional{\frac{a(-j)}{c} + \frac{(-n)}{q_1}}}.
    \]
    Note that $\fractional{-x} = 1 - \fractional{x}$ when $x$ is not an integer. 
    Hence by Lemma \ref{lemma:aj+nrq2-not-divis-by-c}, we have
    \[
Z =     \sum_{j \shortmod c} \sum_{n \shortmod{q_1}} \overline{\chi_2}(j) \overline{\chi_1}(n) \parentheses{1 - {\fractional{\frac{aj}{c} + \frac{n}{q_1}}}}.
    \]
    The part of the sum with "$1$" vanishes by Lemma \ref{lemma:character-orthogonality}.  
    Hence $Z = -Z$, so $Z=0$.
    \end{proof}

\section{Proofs of the main theorems}
\subsection{Initial control of denominator}
As a warm-up, note that
Proposition \ref{proposition:the-half-cancels}
immediately implies $S_{\chi_1, \chi_2}(a,c) \in \frac{1}{c^2} Z[\chi_1, \chi_2]$.  This has the weakness that the denominator grows quadratically in $c$.  As a first step towards improving this, we have the following result which displays linear dependence in $c$.  Actually, it is linear in $c/q_2$, which is even better.
\begin{theorem}
\label{thm:denominator-is-rq1}
    Say $c  = r q_1 q_2$ for some positive integer $r$.  Then
    \[
    S_{\chi_1, \chi_2}(a,c) = \frac{z}{r q_1} \quad \quad \text{for some $z \in \Z[\chi_1, \chi_2]$}.
    \]
\end{theorem}

Theorem \ref{thm:denominator-is-rq1} is analogous to a similar result for the classical Dedekind sum, and our proof is inspired by \cite[Theorem 2]{RG}.
In fact, we have the following more explicit result.
\begin{lemma}
    \label{lemma:new-formula}
    Let $c = r q_1 q_2$ for some positive integer $r$. Then
    \begin{equation}
    \label{eqn:new-formula}
            \dedekindsum(a,c) = -\frac{1}{r q_1} \sum_{j = 0}^{c - 1} \sum_{n = 0}^{q_1 - 1} \overline{\chi_2}(j) \overline{\chi_1}(n) \floor{\frac{j}{q_2}} \floor{\frac{aj}{c} + \frac{n}{q_1}}.
    \end{equation}
\end{lemma}

    Theorem~\ref{thm:denominator-is-rq1} is a direct corollary of Lemma~\ref{lemma:new-formula}.

\begin{proof} [Proof of Lemma~\ref{lemma:new-formula}]
   We use \eqref{eqn:dedekind-sum-fractional}, and 
take $0 \leq j < c$ so that $\left\{\frac{j}{c}\right\} = \frac{j}{c}$, giving
    \[
    \dedekindsum(a,c)=\sum_{j = 0}^{c - 1} \sum_{n = 0}^{q_1 - 1} \overline{\chi_2}(j) \overline{\chi_1}(n) \frac{j}{c} \left\{\frac{aj}{c} + \frac{n}{q_1}\right\}.
    \]
     Substituting 
    $
    \fractional{\frac{aj}{c} + \frac{n}{q_1}} = \left(\frac{aj}{c} + \frac{n}{q_1} \right) - \floor{\frac{aj}{c} + \frac{n}{q_1}} $ 
    yields
    \[
    \dedekindsum(a,c) = P - Q,
    \]
    where
    \[
    P = \sum_{j = 0}^{c - 1} \sum_{n = 0}^{q_1 - 1} \overline{\chi_2}(j) \overline{\chi_1}(n) \left( \frac{a j^2}{c^2} + \frac{nj}{cq_1} \right), \quad 
    Q = \sum_{j = 0}^{c - 1} \sum_{n = 0}^{q_1 - 1} \overline{\chi_2}(j) \overline{\chi_1}(n) {\frac{j}{c}} \floor{\frac{a j}{c} + \frac{n}{q_1}}.
    \]

    We will show that 
    \begin{align}
    \label{eqn:P}
    P = \frac{1}{q_1 q_2} \left[ \sum_{m = 0}^{q_2 - 1} \overline{\chi_2}(m) m \right] \cdot \left[\sum_{n = 0}^{q_1 - 1} \overline{\chi_1}(n) n \right],
    \end{align}
    and
     \begin{align}
     \label{eqn:Q}
    Q = P+\sum_{j_0 = 0}^{q_2 - 1} \sum_{n = 0}^{q_1 - 1} \overline{\chi_2}(j_0) \overline{\chi_1}(n) \sum_{k = 0}^{rq_1 - 1}  {\frac{k q_2}{c}} \floor{\frac{a (j_0 + k q_2)}{c} + \frac{n}{q_1}}.
    \end{align}
    
    First, we focus on $P$. Using Lemma \ref{lemma:character-orthogonality}, the $\frac{aj^2}{c^2}$ term will simply vanish, so that
    \[
    P=\sum_{j = 0}^{c - 1}   \sum_{n = 0}^{q_1 - 1}  \overline{\chi_2}(j)\overline{\chi_1}(n)\frac{nj}{ c q_1}.
    \]
    Let
    $j = j_0 + k q_2$ 
    where $0 \leq j_0 < q_2$ and $0 \leq k < \frac{c}{q_2} = r q_1$, giving
    \begin{align*}
        P&=\sum_{j_0 = 0}^{q_2 - 1}\sum_{n = 0}^{q_1 - 1}  \overline{\chi_2}(j_0)\overline{\chi_1}(n) \sum_{k = 0}^{rq_1 - 1}  
        \parentheses{
        \frac{nj_0}{ c q_1} + \frac{n k q_2}{c q_1}}.
    \end{align*}

    Using Lemma \ref{lemma:character-orthogonality} again, the $\frac{n k q_2}{c q_1}$ term will vanish, so that
    
    \begin{align*}
    P&=\sum_{j_0 = 0}^{q_2 - 1}\sum_{n = 0}^{q_1 - 1} \overline{\chi_2}(j_0) \overline{\chi_1}(n) \sum_{k = 0}^{rq_1 - 1}  \frac{nj_0}{ c q_1} \\
    &=
    \brackets{\sum_{k = 0}^{r q_1 - 1}  \frac{1}{ c q_1}} \cdot
    \brackets{\sum_{j_0 = 0}^{q_2 - 1}\overline{\chi_2}(j_0)j_0} \cdot \brackets{\sum_{n = 0}^{q_1 - 1} \overline{\chi_1}(n) n}.
    \end{align*}
    This directly simplifies to give \eqref{eqn:P}.

    Turning to $Q$, we change variables $j = j_0 + k q_2$ 
    with $0 \leq j_0 < q_2$ and $0 \leq k < \frac{c}{q_2} = rq_1$ as before,
    to obtain
    \[
    Q = \sum_{j_0 = 0}^{q_2 - 1} \sum_{n = 0}^{q_1 - 1} \overline{\chi_2}(j_0) \overline{\chi_1}(n) \sum_{k = 0}^{r q_1 - 1}  {\frac{j_0 + k q_2}{c}} \floor{\frac{a (j_0 + k q_2)}{c} + \frac{n}{q_1}}.
    \]
    We distribute this sum to get $Q = R + T$, where

    \begin{align}
    \label{eqn:R}
    R = \sum_{j_0 = 0}^{q_2 - 1} \sum_{n = 0}^{q_1 - 1} \overline{\chi_2}(j_0) \overline{\chi_1}(n) \sum_{k = 0}^{r q_1 - 1}  {\frac{j_0}{c}} \floor{\frac{a (j_0 + k q_2)}{c} + \frac{n}{q_1}}, 
    \end{align}
    \begin{align}
    \label{eqn:T}
    T = \sum_{j_0 = 0}^{q_2 - 1} \sum_{n = 0}^{q_1 - 1} \overline{\chi_2}(j_0) \overline{\chi_1}(n) \sum_{k = 0}^{r q_1 - 1}  {\frac{k q_2}{c}} \floor{\frac{a (j_0 + k q_2)}{c} + \frac{n}{q_1}}.
    \end{align}
    We will now show $R = P$, which will then imply \eqref{eqn:Q}. 
        We apply Lemma \ref{lemma:floor-function-properties} with $x=\frac{aj_0}{q_2}+rn$, $N=rq_1$, and with both $a$ and $k$ remaining the same. This leaves us with the innermost sum in \eqref{eqn:R} evaluating as
    
    \begin{align*}
    \frac{j_0}{c}\left(\floor{\frac{aj_0}{q_2}+rn} +\frac{1}{2}(a-1)(rq_1-1)\right)&=
    \frac{j_0}{c} \floor{\frac{a j_0}{q_2}} +\frac{nj_0}{q_1q_2}+\frac{j_0(a-1)(rq_1-1)}{2c}.
    \end{align*}
    All the terms except $\frac{nj_0}{q_1q_2}$ depend on only one of $n$ or $j_0$, so by Lemma~\ref{lemma:character-orthogonality} they will vanish, leaving us with $R$ as in \eqref{eqn:P}.

    Using $\dedekindsum(a,c) = P - Q$ and $Q = P + T$ yields
    $\dedekindsum(a,c) =  -T$, and hence
\[
    \dedekindsum(a,c) = -\frac{1}{rq_1} \sum_{j_0 = 0}^{q_2 - 1} \sum_{n = 0}^{q_1 - 1} \overline{\chi_2}(j_0) \overline{\chi_1}(n) \sum_{k = 0}^{r q_1 - 1} k \floor{\frac{a (j_0 + k q_2)}{c} + \frac{n}{q_1}}.
\]

    Reverting the previous change of variables $j = j_0 + k q_2$ and
    using $\floor{\frac{j}{q_2}} = \floor{\frac{j_0 + k q_2}{ q_2}} = k$ then yields \eqref{eqn:new-formula}.
\end{proof}

\subsection{Proof of Theorems~\ref{thm:main-thm-gamma1} and ~\ref{thm:main-thm-gamma0}}
It is useful to state
the following conditions.
\begin{conditions}
    \label{cond:gamma-1}
    Suppose $\gamma \in \Gamma_1(q_1 q_2)$.
\end{conditions}
\begin{conditions}
    \label{cond:gamma-0}
    Suppose $\gamma \in \Gamma_0(q_1 q_2)$, $\chi_1, \chi_2$ are quadratic, and $q_1, q_2 >4$ are odd. 
\end{conditions}

The key to our proof of Theorems~\ref{thm:main-thm-gamma1} and ~\ref{thm:main-thm-gamma0} is the reciprocity formula (Proposition \ref{prop:reciprocity-law}).
The following is a simplified form of the reciprocity law, which allows us to interchange $\gamma$ and $\gamma'$ while retaining the integrality of the sum.
\begin{lemma}
\label{lemma:reciprocity-law-simplification}
Under either Condition~\ref{cond:gamma-1} or~\ref{cond:gamma-0},
\[
    \dedekindsum(\gamma) = \chi_1(-1) S_{\chi_2, \chi_1}(\gamma') + 2z \quad \text{ for some } z \in \Z[\chi_1, \chi_2].
\]
\end{lemma}
\begin{proof}
    Under Condition~\ref{cond:gamma-1}, Lemma \ref{lemma:reciprocity-law-simplification} follows directly from the reciprocity law, indeed, with $z=0$. So, assume Condition~\ref{cond:gamma-0}. If both $\chi_1,\chi_2$ are even, Lemma \ref{lemma:reciprocity-law-simplification} also follows directly from Proposition~\ref{prop:reciprocity-law}, so, suppose $\chi_1, \chi_2$ are odd.
    For a primitive odd quadratic character $\chi$ with odd conductor $q > 4$,
    and corresponding class number $h(-q)$, we recall the well-known class number formula
    \[
        L(1, \chi) = \frac{h(-q) \pi}{\sqrt{q}}.
    \]
    Therefore, Proposition~\ref{prop:reciprocity-law} gives
    \[
    \dedekindsum(\gamma) = -S_{\chi_2, \chi_1}(\gamma') + (1 - \psi(\gamma)) \frac{\tau(\overline{\chi_1}) \tau(\overline{\chi_2})}{(\pi i)^2}
    \parentheses{\frac{h(-q_1) \pi}{\sqrt{q_1}}}
    \parentheses{\frac{h(-q_2) \pi}{\sqrt{q_2}}}.
    \]
For $\chi$ odd and quadratic of conductor $q$, $\tau(\chi) = i \sqrt{q}$.    
    Hence
    \begin{equation}
    \label{eqn:dedekindsum-psi}
    \dedekindsum(\gamma) = -S_{\chi_2, \chi_1}(\gamma') + (1 - \psi(\gamma))) h(-q_1) h(-q_2).
    \end{equation}
    
    Since the class number is an integer and $\psi(\gamma) = \pm 1$ when $\chi_1, \chi_2$ are quadratic, we have that $(1 - \psi(\gamma))$ is either $0$ or $2$. So, in either case, we have shown our desired result.
\end{proof}

\begin{lemma}
    \label{lemma:denominator-q2}
    Under either Condition~\ref{cond:gamma-1} or Condition~\ref{cond:gamma-0},
    \[
        \dedekindsum(\gamma) \in \frac{1}{q_2} \Z[\chi_1, \chi_2].
    \]
\end{lemma}

\begin{proof}
    The following proof will work for both conditions. Let $\gamma = (\begin{smallmatrix}
        a & b \\ c & d
    \end{smallmatrix})$. Using Proposition~\ref{prop:crossedhomomorphism} and that $S_{\chi_1,\chi_2}(1,0)=0$, we have that \[
        S_{\chi_1, \chi_2} \left( \begin{pmatrix}
            a & b \\ c & d
        \end{pmatrix} \right) = S_{\chi_1, \chi_2} \left( \begin{pmatrix}
            a & b \\ c & d
        \end{pmatrix} \begin{pmatrix}
            1 & n \\ 0 & 1
        \end{pmatrix}  \right) = S_{\chi_1, \chi_2} \left( \begin{pmatrix}
            a & an + b \\ c & cn + d
        \end{pmatrix} \right).
    \]
    Then, by Lemma~\ref{lemma:reciprocity-law-simplification},
    \[
        S_{\chi_1, \chi_2} (a, c) = \chi_1(-1) S_{\chi_2, \chi_1} (cn+d, -(an+b)q_1 q_2) + 2z \quad \text{for some }z \in \Z[\chi_1, \chi_2].
    \]
    By Theorem~\ref{thm:denominator-is-rq1}, we have that
    \begin{equation}
        \label{eqn:q2-an+b}
        \dedekindsum(a,c) \in \frac{1}{q_2 (an + b)} \Z[\chi_1, \chi_2]  + 2\Z[\chi_1, \chi_2]= \frac{1}{q_2 (an + b)} \Z[\chi_1, \chi_2].
    \end{equation}
Here $n$ is an arbitrary integer at our disposal.  
    Using Dirichlet's theorem on arithmetic progressions, we may choose two distinct primes $p_1$ and $p_2$ of the form $an + b$ such that $p_1, p_2 > q_2$.  
    Then
    \begin{equation}
        \dedekindsum(a,c) \in \frac{1}{p_1 q_2} \Z[\chi_1, \chi_2] \cap \frac{1}{p_2 q_2} \Z[\chi_1, \chi_2] = \frac{1}{q_2} \Z[\chi_1, \chi_2].
    \end{equation}
This completes the proof.
\end{proof}

\begin{corollary}
 
    \label{cor:denominator-gcd}
    Under either Condition~\ref{cond:gamma-1} or Condition~\ref{cond:gamma-0},
    \[
        \dedekindsum(\gamma) \in \frac{1}{gcd(q_1, q_2)} \Z[\chi_1, \chi_2].
    \]
\end{corollary}
\begin{proof}
    By Lemma \ref{lemma:denominator-q2}, we have that
    \[
        \dedekindsum(\gamma) \in \frac{1}{q_2} \Z[\chi_1, \chi_2].
    \]
    By Lemma \ref{lemma:reciprocity-law-simplification} and Lemma \ref{lemma:denominator-q2}, we also have that
    \[
        \dedekindsum(\gamma) \in \frac{1}{q_1} \Z[\chi_1, \chi_2].
    \]
    Intersecting these two sets yields the corollary. 
\end{proof}
\begin{proof}[Proof of Theorems~\ref{thm:main-thm-gamma1} and ~\ref{thm:main-thm-gamma0}]
    These theorems directly follow from Corollary~\ref{cor:denominator-gcd} and the fact that $\Z[\chi_1, \chi_2] = \Z$ under Condition~\ref{cond:gamma-0}.
\end{proof}

\section{Acknowledgements}
This study took place in summer 2024 as part of an REU program at Texas A\&M University. The authors extend their gratitude to the Department of Mathematics at Texas A\&M University and the NSF for their support of the REU. 
The authors thank
Agniva Dasgupta (Texas A\&M University) for his lectures and feedback during the REU. 
Finally, the first three named authors would like to thank Reveille the dog for providing them with much needed emotional support during this project.

\bibliographystyle{alpha}
\bibliography{references}	

\end{document}